\documentclass[conference]{IEEEtran}
\IEEEoverridecommandlockouts
\usepackage[utf8]{inputenc}
\usepackage[T1]{fontenc}
\usepackage{amsmath, amssymb, amsthm}
\usepackage{graphicx}
\usepackage[english]{babel}
\usepackage[pageanchor=false,colorlinks=true,linkcolor=blue,citecolor=blue,urlcolor=blue]{hyperref}
\usepackage{xcolor}

\usepackage{subfigure}
\usepackage{algorithm, algorithmicx, algpseudocode}

\usepackage{mathtools}
\DeclareMathOperator*{\argmax}{\arg\max}
\DeclareMathOperator*{\argmin}{\arg\min}

\newtheorem{defintion}{Definition}
\newtheorem{lemma}{Lemma}

\begin{document}
	\title{Projected Block Coordinate Descent for sparse spike estimation}
	\author{\IEEEauthorblockN{Pierre-Jean Bénard\textsuperscript{1,*}, Yann Traonmilin\textsuperscript{1}, Jean-François Aujol\textsuperscript{1}
			\\
			\textit{\textsuperscript{1}Univ. Bordeaux, Bordeaux INP, CNRS,  IMB, UMR 5251,F-33400 Talence, France}}
		\textsuperscript{*}contact: pierre-jean.benard@math.u-bordeaux.fr
	}

	\maketitle

	\begin{abstract}
		We consider the problem of recovering off-the-grid spikes from linear measurements. The state of the art Over-Parametrized Continuous Orthogonal Matching Pursuit (OP-COMP) with Projected Gradient Descent (PGD) successfully recovers those signals. In most cases, the main computational cost lies in a unique global descent on all parameters (positions and amplitudes). In this paper, we propose to  improve this algorithm by accelerating this descent step.
		We introduce a new algorithm, based on  Block Coordinate Descent, that takes advantages of the sparse structure of the problem. Based on qualitative theoretical results, this algorithm shows improvement in calculation times in realistic synthetic microscopy experiments.
	\end{abstract}

	\section{Introduction}
	The  off-the-grid super-resolution problem consists in recovering spikes from linear measurements in a continuous setting. It has applications in many fields. Among these applications, one can cite Single Molecule Localization Microscopy (SMLM) or Magnetic Resonance Imaging (MRI). In these fields, the signals to recover display properties of sparsity that can be exploited for a faster recovery.

	Let $ x_{0} $ be an off-the-grid sparse signal over $ \mathbb{R}^{d} $. Such signals can be modeled as a sum of $ K $ Dirac measures:
	\begin{equation}\label{eq:x0}
		x_{0} = \sum_{i=1}^{K} a_{i} \delta_{t_{i}}
	\end{equation}
	where $ a = (a_{1}, \dots, a_{K}) \in \mathbb{R}^{K} $ are the amplitudes and $ t = (t_{1}, \dots, t_{K}) \in  \mathbb{R}^{K \times d} $ are the positions of the spikes.
	We observe the signal through a linear operator $ A $ from the space $ \mathcal{M} $ of finite signed measures over $ \mathbb{R}^{d} $ to $ \mathbb{C}^{m} $ with $ m $ the number of measurements. The result of this observation is $ y \in \mathbb{C}^{m} $ which can be written as $ y = A x_{0} $.
	Note that we consider the noiseless case in this article for the sake of clarity.

	A way to recover the true signal $ x_{0} $ from its observation is to find the minimizer of a non-convex least-square problem:
	\begin{equation}\label{eq:func_to_minimize}
		x^{*} \in \argmin_{x \in \Sigma_{K, \epsilon}} \| A x - y \|_{2}^{2}
	\end{equation}
	where $\Sigma_{K, \epsilon}$, defined in \eqref{eq:set_separation_constraint}, is a set modeling a separation constraint between spikes.
	We consider the set of parameters $\Theta_{K, \epsilon}$,
	\begin{multline}
		\Theta_{K, \epsilon} = \left\{ \theta = (a_{1}, \dots, a_{K}, t_{1}, \dots, t_{K}) \in \mathbb{R}^{K(d + 1)}, \right. \\
		\left. \forall i, j \in \{ 1, \dots, K \}, i \neq j, \| t_{i} - t_{j} \|_{2} > \epsilon \right\}.
	\end{multline}
	We can rewrite \eqref{eq:x0} with the variable $\theta$:
	\begin{equation}
		x_{0} = \phi(\theta) := \sum_{i = 1}^{K} a_{i} \delta_{t_{i}} \in \Sigma_{K, \epsilon}.
	\end{equation}
	The unknown $x_{0}$ belongs to the low dimensional model
	\begin{equation}\label{eq:set_separation_constraint}
		\Sigma_{K, \epsilon} = \phi(\Theta_{K, \epsilon})
	\end{equation}
	The problem \eqref{eq:func_to_minimize} can then be written as
	\begin{equation} \label{eq_g}
		\theta^{*} \in \argmin_{\theta \in \Theta_{K, \epsilon}} g(\theta) \quad \text{with} \quad g(\theta) := \| A \phi(\theta) - y \|_{2}^{2}.
	\end{equation}
	Theoretical guarantees for the recovery of $x_{0}$ with \eqref{eq:func_to_minimize} have been given in \cite{gribonval2021compressive}. As an example, when the linear operator $A$ models  random Fourier measurements at frequencies drawn with a Gaussian distribution and $\dim(y) \gtrapprox \mathcal{O}(K^{2} d \text{polylog}(K, d))$, we have $x^{*} = x_{0}$.
	In practice, Sliding Continuous Orthogonal Matching Pursuit (SCOMP) \cite{keriven2017compressive} has been successful at minimizing \eqref{eq:func_to_minimize}. Recovery guarantees are given in \cite{elvira2019omp, elvira2019does,benard2023estimation}. We can also cite the  Sliding Frank-Wolfe (SFW) algorithm \cite{denoyelle:hal-01921604:sliding-frank-wolfe}, that aims to solve a regularized version of \eqref{eq:func_to_minimize}.
	Although state-of-the-art methods can fully recover signals in this setting, they lack of performance in computation time. The Over-Parametrized COMP + Projected Gradient Descent (OP-COMP + PGD) algorithm has been introduced to address these performance issues~\cite{traonmilin:projected_gradient_descent,benard2022fast}. This method  initializes an over-parametrized signal without a costly descent step (OP-COMP). Then, it performs a single descent on all parameters while performing a  projection to reduce the number of spikes (PGD).
	It has been proven that a good initialization of the signal is possible in $ K $ steps \cite{benard2023estimation}. This means that each initialized spike is  close to a true spike, on the conditions that one has a sufficient separation $ \epsilon $ and a well-behaved operator $ A $. However, these conditions are not always verified. In these cases, while OP-COMP + PGD is still faster than SFW and SCOMP, it lacks in efficiency during the descent on all parameters in $ \theta $. The large number of gradients to compute is the main limiting factor of OP-COMP + PGD~\cite{benard2023estimation}.

	For gradients descent methods, an idea to accelerate computation time is the Block Coordinate descent (BCD). It consists in only performing iterations of the descent on a subset of the coordinates of the  parameters that we estimate.

	The idea of an underlying structure to separate the signal in blocks is  common in many applications (see e.g. \cite{shi2016primer, wright2015coordinate}). Among many applications, there are Non-Negative Matrix Factorization \cite{xu2013block}, Sparse Logistic Regression \cite{wu2008coordinate}. The LASSO problem has also been solved with BCD methods \cite{li2009coordinate, wright2015coordinate}.

	The definition of Block Coordinate Descent relies on the blocking strategy, the block selection rule and the block update rule \cite{nutini2022lets}.

	\noindent \textit{Blocking strategy.} The blocks can be either fixed beforehand or variable and recomputed at each iteration. Fixed blocks permit better control and separation in sub-problems \cite{sardy2000block} when the structure of the problem is known before computations \cite{meier2008group}. Variable blocks allow more freedom, e.g. by creating blocks based on greedy methods \cite{tseng2009coordinate} such as Gauss-Southwell.

	\noindent \textit{Block selection rule.}
	As the choice of the block to update may lead to different minima, it is a main area of focus for BCD methods \cite{shi2016primer}. Basic ideas such as the cyclic variants for index choice are common as they are deterministic methods with proven convergence \cite{beck2013convergence}. Implementation-wise, they are simple and do not require computation overhead.
	Instead of going through all coordinates in a cyclic order, one can use randomized methods by choosing coordinates based on a uniform distribution \cite{nesterov2012efficiency}, weighted distribution \cite{allen2016even} or just from arbitrary rules \cite{richtarik2014iteration}.
	Greedy methods are also used in the block selection process. They choose the most suitable block at each iteration by solving an optimization problem \cite{nutini2022lets}. A well known greedy method, for its efficiency both theoretically and practically and its simplicity, is the Gauss-Southwell (GS) rule which select a block based on the norm of its gradient,
	\begin{equation}\label{eq:gauss-southwell}
		b_{k} \in \argmax_{b \in \mathcal{B}} \| \left[ \nabla g(\theta^{k}) \right]_{b} \|_{2}.
	\end{equation}
	where $ \mathcal{B} $ is the set of all blocks. Other greedy methods such as Maximum Improvement \cite{chen2012maximum} or GS-Lipschitz \cite{nutini2015coordinate} are also valid choices for block selection rules.
	Greedy methods perform better than cyclic or random methods but at the cost of a more expensive computation \cite{nutini2015coordinate}.

	\noindent \textit{Block update rule.} The role of this rule is to make the chosen block maximally decrease the objective function (like \eqref{eq:func_to_minimize}). It is important to use a block update rule suited with the problem. Starting with the classical gradient descent, one can set the step-size in different ways (fixed, approximate line-search or optimal step-size) to improve performance \cite{bertsekas1997nonlinear}. Many more descent methods used in convex and non-convex optimization are used in BCD. Among the many methods introduced, one can use proximal update \cite{auslender1992asymptotic}, stochastic gradient \cite{nemirovski2009robust} depending, for example, on the geometry of the objective function.
	In our case, we wish to improve upon the Projected Gradient Descent from OP-COMP + PGD using a FISTA Restart method for the block updating rule. An early use of FISTA for BCD method can be found in \cite{qin2013efficient}. While there are no theoretical guarantees for the convergence of FISTA in BCD methods, it works well in various cases \cite{chambolle2015convergence}.


	\noindent\textbf{Contributions.}
	In this paper, our main contribution is an accelerated method for recovering sparse signals based on OP-COMP + PGD.
	We propose a Projected Block Coordinate Descent method that performs the descent step on spikes that have not already converged. The main ingredient for this BCD is a block selection strategy based on the specific structure of the sparse spike model to decompose the problem in small blocks. We define the three ingredients defining  BCD within the OP-COMP + PGD method and justify qualitatively our choices.
	We provide synthetic experiments in the context of microscopy that show an improvement up to $ 35 \% $ in calculation time compared to PGD.

	\section{The Projected Block Coordinate Descent algorithm}

	In this section, we motivate the idea of using a block system to divide our problem in a sum of smaller and simpler problems. Moreover, we describe our method: the projected Block Coordinate Descent. This descent only operates on parameters of non-converged spikes of the estimated signal.

	\subsection{A result to guide block selection}

	The Block Coordinate Descent (BCD) is a way to fasten the computation time of descent algorithms by using the underlying structure of the problem. In our case, the more spikes there are in the signal to recover, the more parameters there are to compute in the descent process.  As our model contains a separation constraint, every spike to be recovered is estimated by one or more initialized spikes independently of the other true spikes. This leads to two cases: either there is one initialized spike estimating a true spike, or a cluster of initialized spikes. For one spike estimations, the convergence to their associated true spike is fast. On the contrary for the clusters of initialized spikes, their convergence is much slower toward a true spike~\cite{benard2023estimation}. We propose to only compute the gradients of spikes that have not yet converged. Next, we show qualitatively why we can use our model to separate the recovery of each spike using our separation constraint with an idealized initialized signal.

	Let $x_{0} = \sum_{i = 1}^{K} a_{i} \delta_{t_{i}} \in \Sigma_{K, \epsilon}$ the signal to recover and $x = \sum_{j = 1}^{K} b_{j} \delta_{s_{j}}$ the initialized signal obtained with COMP. Suppose that $x$ is well initialized i.e. for all $i \in \{ 1, \dots, K \}, \| t_{i} - s_{i} \|_{2} < \frac{\epsilon}{3}$.
	By computing their difference, we get
	\begin{equation}
		x_{0} - x = \sum_{i = 1}^{K} \left( a_{i} \delta_{t_{i}} - b_{i} \delta_{s_{i}} \right) = \sum_{i = 1}^{K} \nu_{i}
	\end{equation}
	where $\nu_{i}= a_{i} \delta_{t_{i}} - b_{i} \delta_{s_{i}} $ is called a $\frac{\epsilon}{3}$-dipole. From \cite{benard2023estimation}, we recall the definition of a $\frac{\epsilon}{3}$-dipole,
	\begin{defintion}[($\epsilon$-)Dipole]
		An $\epsilon$-dipole is a measure $ \nu = a \delta_{t} - b \delta_{s} $ where $ \| t - s \|_{2} \leq \epsilon $.
	\end{defintion}
	\noindent Let us also recall the definition of $ \frac{\epsilon}{3} $-separated dipoles,
	\begin{defintion}[($\epsilon$)-Separation of dipoles]
		Two dipoles $ \nu_{1} = a_{1} \delta_{t_{1}} - b_{1} \delta_{s_{1}} $ and $ \nu_{2} = a_{2} \delta_{t_{2}} - b_{2} \delta_{s_{2}} $ are $ \epsilon $-separated if their support is strictly $ \epsilon $-separated, i.e. if $ \| t_{1} - t_{2} \|_{2} > \epsilon $ and $ \| t_{1} - s_{2} \|_{2} > \epsilon $ and $ \| s_{1} - t_{2} \|_{2} > \epsilon $ and $ \| s_{1} - s_{2} \|_{2} > \epsilon $.
	\end{defintion}
	\noindent The following lemma gives us bounds to better interpret the block structure of our model with our initialization $ x $.
	\begin{lemma}\label{lem:bounds_y-Ax}
		Let $ x_{0} = \sum_{i = 1}^{K} a_{i} \delta_{t_{i}} \in \Sigma_{K, \epsilon} $ and $ x = \phi(\theta) = \sum_{i = 1}^{K} b_{i} \delta_{s_{i}} $ with $ \theta = (b_{1}, \dots, b_{K}, s_{1}, \dots, s_{K}) $ such that for all $ i \in \{ 1, \dots, K \} $, $ \| t_{i} - s_{i} \|_{2} < \frac{\epsilon}{3} $. Let $ A $ be a linear operator from the space of finite measures over $ \mathbb{R}^{d} $ to $ \mathbb{C}^{m} $. Suppose that for two $ \frac{\epsilon}{3} $-separated $ \frac{\epsilon}{3} $ dipoles $ \nu_{1}, \nu_{2} $, there exists $ \mu_{A} > 0 $ such that $ A $ follows the assumption,
		\begin{equation}\label{ass:mutual_coherence_A}
			\left\langle A \nu_{1}, A \nu_{2}  \right\rangle \leq \mu_{A} \left\| A \nu_{1} \right\|_{2}^{2} \left\| A \nu_{2} \right\|_{2}^{2}.
		\end{equation}
		Then we obtain
		\begin{equation}\label{eq:lem_res}
			\left\| y - A x \right\|_{2}^{2} \leq \sum_{i = 1}^{K} \left\| A \nu_{i} \right\|_{2}^{2} + \mu_{A} \sum_{\substack{i, j = 1 \\ i \neq j}}^{K} \left\| A \nu_{i} \right\|_{2}^{2} \left\| A \nu_{j} \right\|_{2}^{2}.
		\end{equation}
		Moreover, let $ \partial_{i, r} g(\theta) $ be the partial derivative of $ g(\theta) $ with respect to the $ r $-th coordinate of $ s_{i} $.
		Then we get
		\begin{multline}\label{eq:lem_res_grad}
			\left| \partial_{i, r} \| y - A x \|_{2}^{2} - \partial_{i, r} \| A \nu_{i} \|_{2}^{2} \right| \leq \\
			2 a_{i} (K - 1) \mu_{A} \| \partial_{i, r} A \nu_{i} \|_{2}^{2} \max_{j \in \{1, \dots, K\}} \| A \nu_{j} \|_{2}^{2}.
		\end{multline}
	\end{lemma}
	\begin{proof}
		Since we suppose that for all $ i \in \{ 1, \dots, K \} $, $ \| t_{i} - s_{i} \|_{2} < \frac{\epsilon}{3} $, we can write $ x_{0} - x = \sum_{i = 1}^{K} \left( a_{i} \delta_{t_{i}} - b_{i} \delta_{s_{i}} \right) $ as a sum of $ \frac{\epsilon}{3} $-dipoles $ \sum_{i = 1}^{K} \nu_{i} $. To get \eqref{eq:lem_res}, we first start by expanding $ \| y - Ax \|_{2}^{2} $,
		\vspace{-10pt}
		\begin{align}
			\left\| y - A x \right\|_{2}^{2} & = \left\| A (x_{0} - x) \right\|_{2}^{2} = \left\| A \sum_{i = 1}^{K} \nu_{i} \right\|_{2}^{2} \\
			& = \sum_{i = 1}^{K} \left\| A \nu_{i} \right\|_{2}^{2} + \sum_{\substack{i, j = 1 \\ i \neq j}}^{K} \left\langle A \nu_{i}, A \nu_{j} \right\rangle
		\end{align}
		We use our assumption \eqref{ass:mutual_coherence_A} so we can bound $ \langle A \nu_{i}, A \nu_{j} \rangle $,
		\vspace{-5pt}
		\begin{equation}
			\left\| y - A x \right\|_{2}^{2} \leq \sum_{i = 1}^{K} \left\| A \nu_{i} \right\|_{2}^{2} + \mu_{A} \sum_{\substack{i, j = 1 \\ i \neq j}}^{K} \left\| A \nu_{i} \right\|_{2}^{2} \left\| A \nu_{j} \right\|_{2}^{2}.
		\end{equation}
		Now, we prove~\eqref{eq:lem_res_grad}. Using Proposition $2.1$ in \cite{traonmilin:hal-01938239}, we have
		\vspace{-5pt}
		\begin{align}
			\partial_{i, r} \| y - A x \|_{2}^{2} = \ & 2 a_{i} \mathcal{R}e \left\langle A \partial_{i, r} \delta_{s_{i}}, A \sum_{j = 1}^{K} \nu_{j} \right\rangle \\
			= \ & 2 a_{i} \mathcal{R}e \left\langle A \partial_{i, r} \delta_{s_{i}}, A \nu_{i} \right\rangle \nonumber \\
			& + 2 a_{i} \mathcal{R}e \left\langle A \partial_{i, r} \delta_{s_{i}}, A \sum_{\substack{j = 1 \\ j \neq i}}^{K} \nu_{j} \right\rangle.
		\end{align}
		Using our assumption \eqref{ass:mutual_coherence_A}, we get (using the fact that this identity is also verified for Dirac derivatives  that are generalized dipoles~\cite{traonmilin:hal-01938239}):
		\begin{align}
			\partial_{i, r} \| y - A x \|_{2}^{2} \leq \ & 2 a_{i} \mathcal{R}e \left\langle A \partial_{i, r} \delta_{s_{i}}, A \nu_{i} \right\rangle \nonumber \\
			& + 2 a_{i} \mu_{A} \| \partial_{i, r} A \delta_{s_{i}} \|_{2}^{2} \sum_{\substack{j = 1 \\ j \neq i}}^{K}  \| A \nu_{j} \|_{2}^{2}.
		\end{align}
		We bound each $ \| A \nu_{j} \|_{2}^{2} $ by their maximum $ \max_{j \in \{1, \dots, K\}} \| A \nu_{j} \|_{2}^{2} $. Moreover, we have that $ 2 a_{i} \mathcal{R}e \left\langle A \partial_{i, r} \delta_{s_{i}}, A \nu_{i} \right\rangle = \partial_{i, r} \| A \nu_{i} \|_{2}^{2} $. Finally, we obtain
		\begin{multline}
			\left| \partial_{i, r} \| y - A x \|_{2}^{2} - \partial_{i, r} \| A \nu_{i} \|_{2}^{2} \right| \leq \\
			2 a_{i} \mu_{A} (K - 1) \| \partial_{i, r} A \nu_{i} \|_{2}^{2} \max_{j \in \{1, \dots, K\}} \| A \nu_{j} \|_{2}^{2}.
		\end{multline}
	\end{proof}
	Using the result \eqref{eq:lem_res} from Lemma \ref{lem:bounds_y-Ax}, we get that in order to minimize $\| y - A x \|_{2}^{2} $, a sufficient condition is to minimize individually each $\| A \nu_{i} \|_{2}^{2} $, i.e. recover $K$ distinct $1$-spike signals.

	To distinguish converged spikes from non-converged spikes, we will compare their gradients using the result \eqref{eq:lem_res_grad} from Lemma \ref{lem:bounds_y-Ax}. Indeed, the absolute difference between $ \partial_{i, r} \| y - A x \|_{2}^{2} $ and $ \partial_{i, r} \| A \nu_{i} \|_{2}^{2} $ is small: For common observation operators (e.g. Gaussian filter), we have that the mutual coherence parameter $ \mu_{A} $ is small and the terms $ \| \partial_{i, r} A \nu_{i} \|_{2}^{2} $ and $ \sup_{j \in \{1, \dots, K\}} \| A \nu_{j} \|_{2}^{2} $ are bounded. In practice, the over-parametrization of the initialized signal yields more spikes than dipoles, however the simpler case given by Lemma \ref{lem:bounds_y-Ax} suggests that blocks of spikes with large gradients should be considered for the selection strategy.

	\subsection{Description of the algorithm}

	Our new algorithm: the Projected Block Coordinate  Descent (Algorithm \ref{alg:BCD}) is an improvement of the Projected Gradient Descent \cite{traonmilin:projected_gradient_descent}, which performs a descent on all parameters while enforcing a separation constraint on the positions of the estimated spikes.
	We keep this main feature and we integrate our Block Coordinate Descent acceleration.  Our method is  derived from the Gauss-Southwell rule. We choose our blocks from the set of all blocks $ \mathcal{B}_{\mathtt{super-res}} $, where each block represents a spike. In each iteration of the main loop, we compute the norm of the gradients ($\| \nabla [ g(\theta) ]_{i} \|_{2}$ for all $i \in \{1, \dots, K\}$) associated to each spike of the estimated signal with $ \nabla [ g(\theta) ]_{i} := \nabla g((a_{i}, t_{i})) $ i.e. the gradients for the amplitude and position of the spike $ i $. Then, we create two sets $I$ and $J$. The set $I = \left\{i, \| \nabla [ g(\theta) ]_{i} \|_{2} \geq \texttt{threshold} \times \max_{j} \| \nabla [ g(\theta) ]_{j} \|_{2} \right\}$ represents all spikes that have not yet converged. In practice, we set manually the value of \texttt{threshold} and we discuss this choice in more details in the next section. We set $J = \{1, \dots, K\} \backslash I$. We note $x_{I} = \sum_{i \in I} a_{i} \delta_{s_{i}}$ and $x_{J} = \sum_{j \in J} a_{j} \delta_{s_{j}}$.
	With only the selected spikes from $I$, we perform a projected gradient descent on $x_{I}$ for a fixed number of iteration $ \mathtt{m_{it}} $ (or until convergence). Note that the observation in this partial projected descent $y_{I}$ is $ y - A x_{J} $ i.e. the original observation updated with the non-selected spikes from this iteration.
	If a projection occurs during this descent, we stop the projected gradient descent. We merge both signals $x_{I}$ and $x_{J}$ and we compute again $I$ and $J$.
	We continue this process until convergence of our estimated signal $x$.
	To perform the descent, we choose the FISTA Restart algorithm \cite{o2015adaptive} as in practice it outperforms a classic Gradient Descent and it provides fast convergence.
	\begin{algorithm}[]
		\caption{Projected Block Coordinate  Descent.}
		\label{alg:BCD}
		\begin{algorithmic}
			\Procedure{BCD}{$ y, A, \theta^{(0)}, \mathtt{n_{it}}, \mathtt{m_{it}} \mathtt{threshold} $}

			\State $ n \gets 0 $
			\State $ k \gets \#\mathtt{spikes} $

			\While{$ n < n_{it} $}
				\State $ I = \{i, \left\| \nabla [ g(\theta^{(n)}) ]_{i} \right\|_{2} \geq $ \\
				\hspace*{\fill} $ \texttt{threshold} \times \max_{j} \left\| \nabla [ g(\theta^{(n)}) ]_{j} \right\|_{2} \} $

				\State $ J = \{1, \dots, k\} \ \backslash \ I $

				\State $ \tilde{\theta}^{(n)} \gets \left\{ \theta_{i}^{(n)}, i \in I \right\} $, $\quad \bar{\theta}^{(n)} \gets \left\{ \theta_{j}^{(n)}, j \in J \right\} $

				\State $\tilde{y} \gets y - A \phi(\bar{\theta}^{(n)}) $
				\vspace{3pt}
				\State $ \tilde{\theta}^{(n^{*})} \gets \mathtt{projected\_descent}(\tilde{y}, \tilde{\theta}^{(n)}, \mathtt{m_{it}}) $

				\State $ \theta^{(n^{*})} \gets \left\{ \tilde{\theta}_{i}^{(n^{*})} \text{ if } i \in I \text{ else } \bar{\theta}_{i}^{(n)} \right\}_{i = 1, \dots, k} $

				\State $ \theta^{(n)} \gets \theta^{(n^{*})} $
				\State $ k \gets \#\mathtt{spikes} $

			\EndWhile

			\State \textbf{return} $ \theta^{(n_\mathtt{it})} $

			\EndProcedure
		\end{algorithmic}
	\end{algorithm}

	\section{Experiments}

	In this section, we compare our new Projected Block Coordinate  Descent with the state-of-the-art Projected Gradient Descent. The signal to recover in every experiment is composed of $50$ Dirac measures for MA-TIRF in $3$ dimensions \cite{denoyelle:hal-01921604:sliding-frank-wolfe}. As these methods aim to recover a signal from an initial estimation, we use Over-Parametrized Continuous Orthogonal Matching Pursuit (OP-COMP) for the initialization of the signal.

	The linear operator used for the observations is the MA-TIRF operator \cite{denoyelle:hal-01921604:sliding-frank-wolfe}. It is used in microscopy for recovering cells in $3$D images. The positions $t$ are in a domain of size $6.4$mm$\times 6.4$mm$\times 0.8$mm. The observation is composed of $4$ planes which represent different focal angles. Each plane is composed of $64 \times 64$ pixel. In total, $\dim(y) = 64 \times 64 \times 4 = 16\,384$. The amplitudes follow a uniform distribution $U([1, 2])$.
	The code for these experiments is available for download at \cite{benard2024code_bcd}.

	\begin{figure}[htbp]
		\centering
		\subfigure[]{
			\includegraphics[width=0.51\columnwidth]{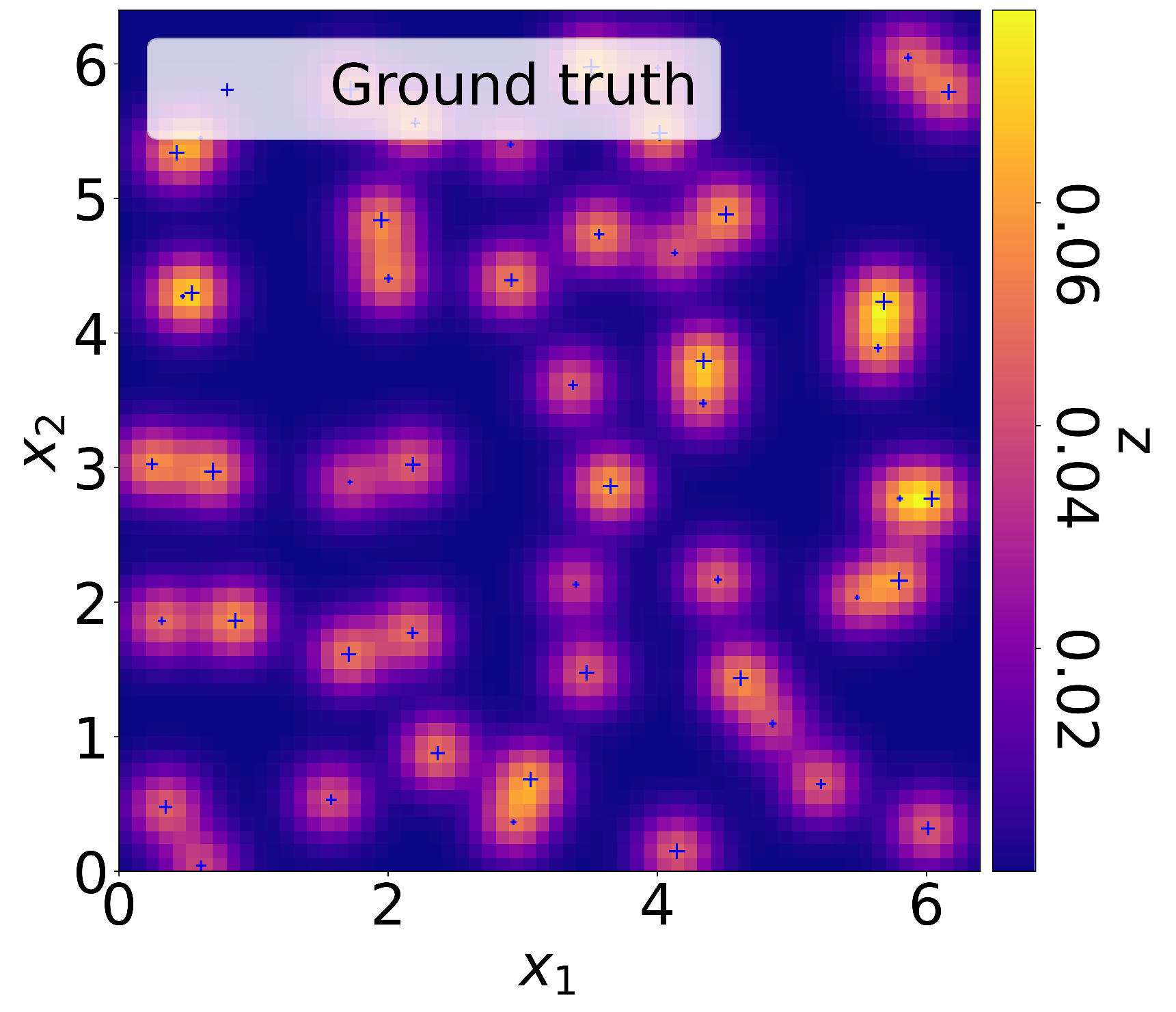}
			\label{fig:MATIRF_y0_xtrue}
		}
		\hspace{-10pt}
		\subfigure[]{
			\includegraphics[width=0.43\columnwidth]{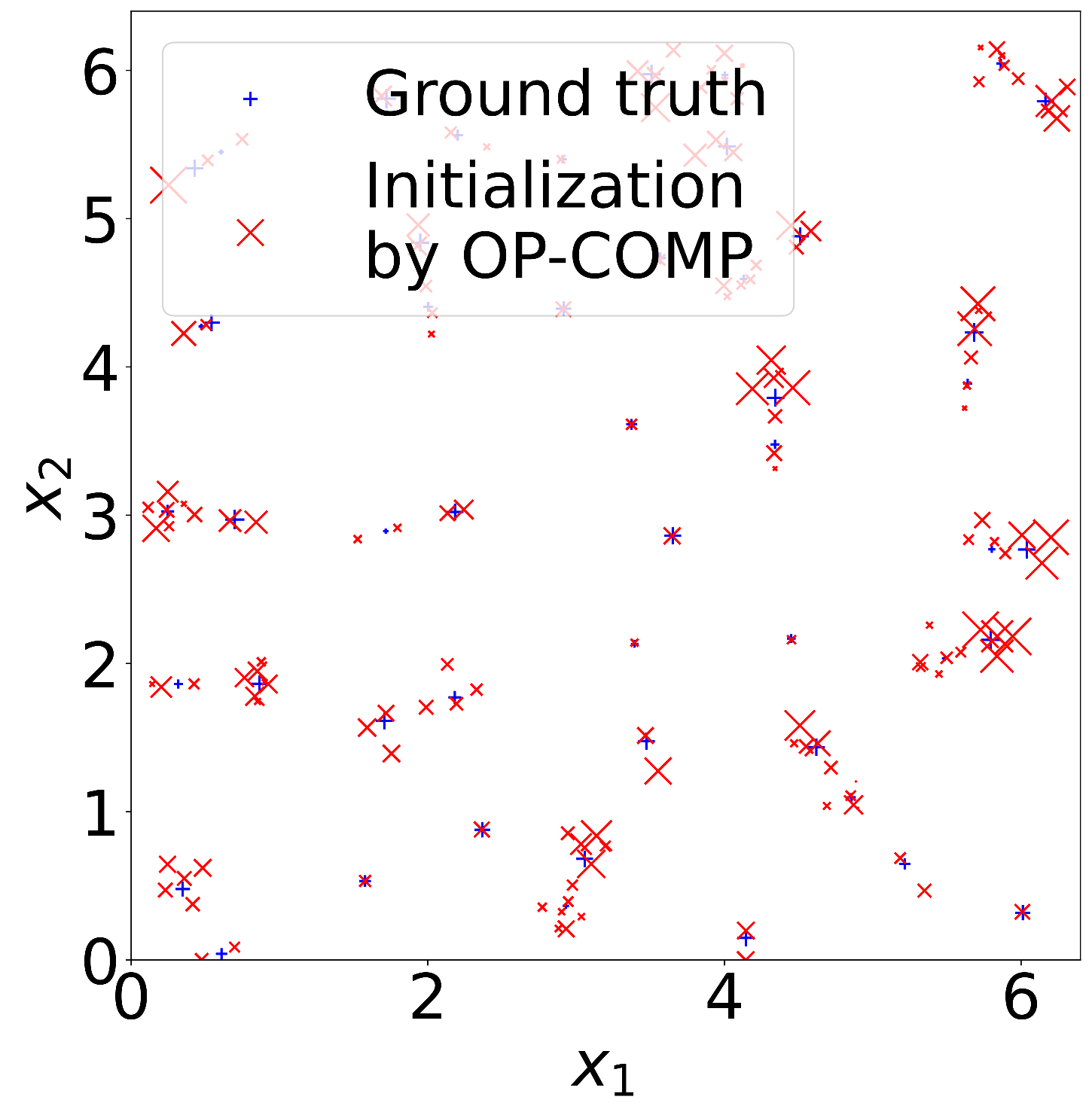}
			\label{fig:MATIRF_xinit}
		}
		\vspace{-.5\baselineskip}
		\caption{
			(a) Ground truth and a plane of its observation through MA-TIRF,
			(b) Ground truth and the initialized signal by OP-COMP.}
		\label{fig:MATIRF_x0}
	\end{figure}

	In Fig. \ref{fig:MATIRF_y0_xtrue}, we display a focal plane of the observation. It corresponds to an image taken at a certain angle. Though it is in $2$ dimensions, the combination of multiple angles is useful to recover spikes in $3$D. We see in this figure that most of the spikes are well separated.
	In Fig. \ref{fig:MATIRF_xinit}, we have the true signal and the initialized signal by OP-COMP. Note that every spike from the ground truth has at least one initialized spike close to itself. Moreover, the over-parametrization is clearly visible as the number of initialized spikes is larger than the number of spikes to recover.

	\begin{figure}[htbp]
		\centering
		\subfigure[]{
			\includegraphics[width=0.46\columnwidth]{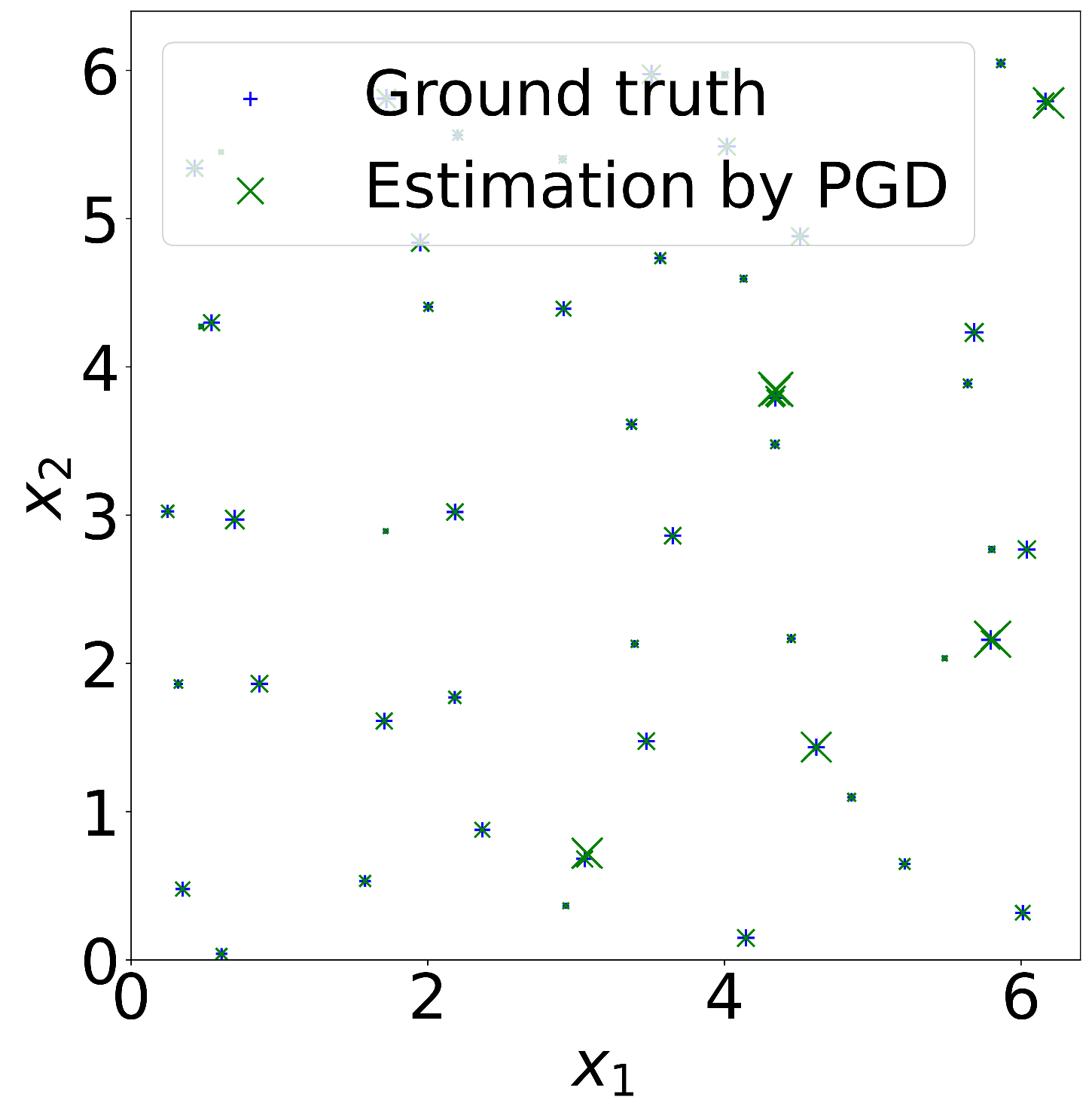}
			\label{fig:MATIRF_xesti_PGD}
		}
		\subfigure[]{
		\includegraphics[width=0.46\columnwidth]{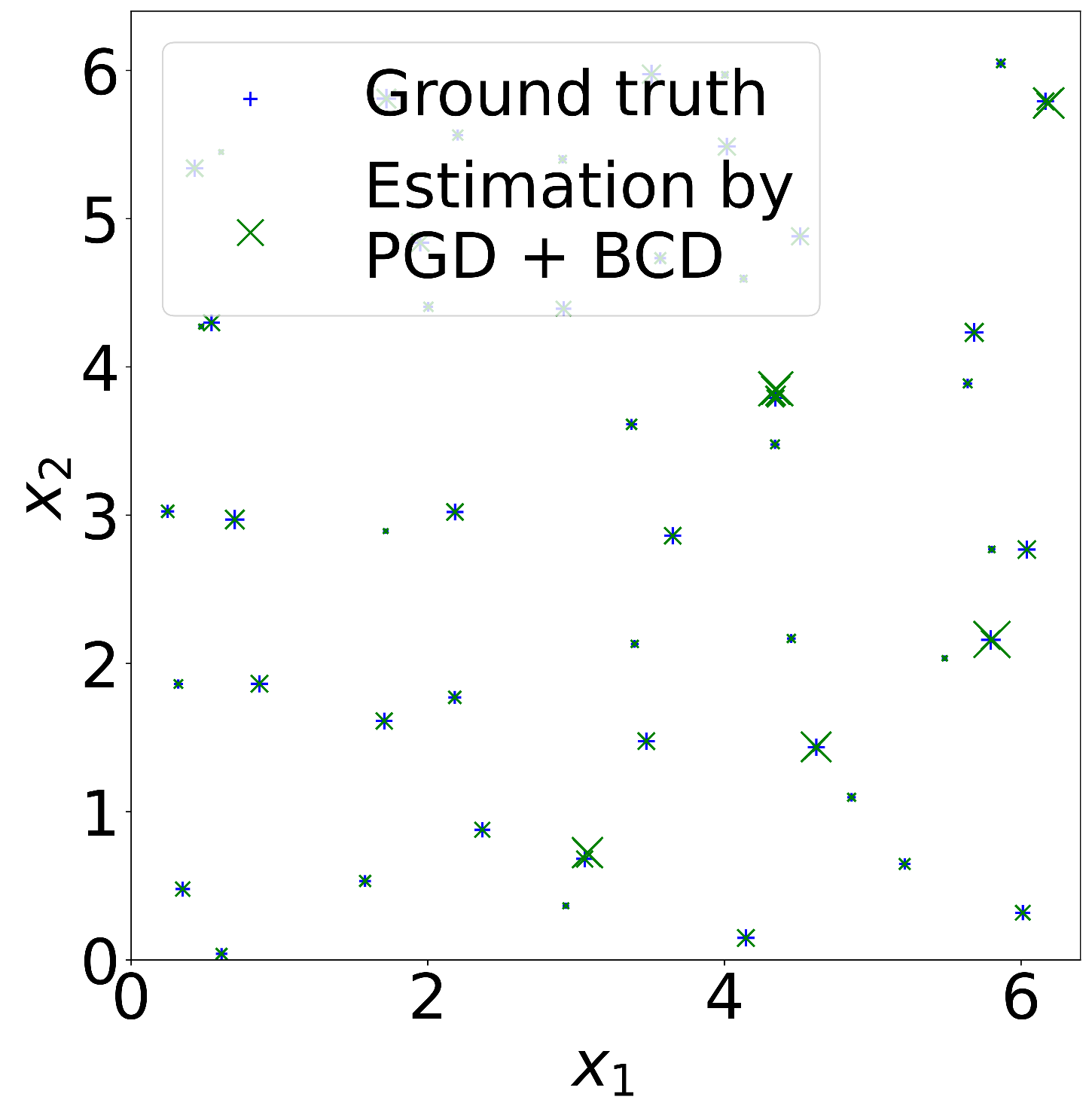}
		\label{fig:MATIRF_xesti_BCD}
		}
		\subfigure[]{
		\includegraphics[width=\columnwidth]{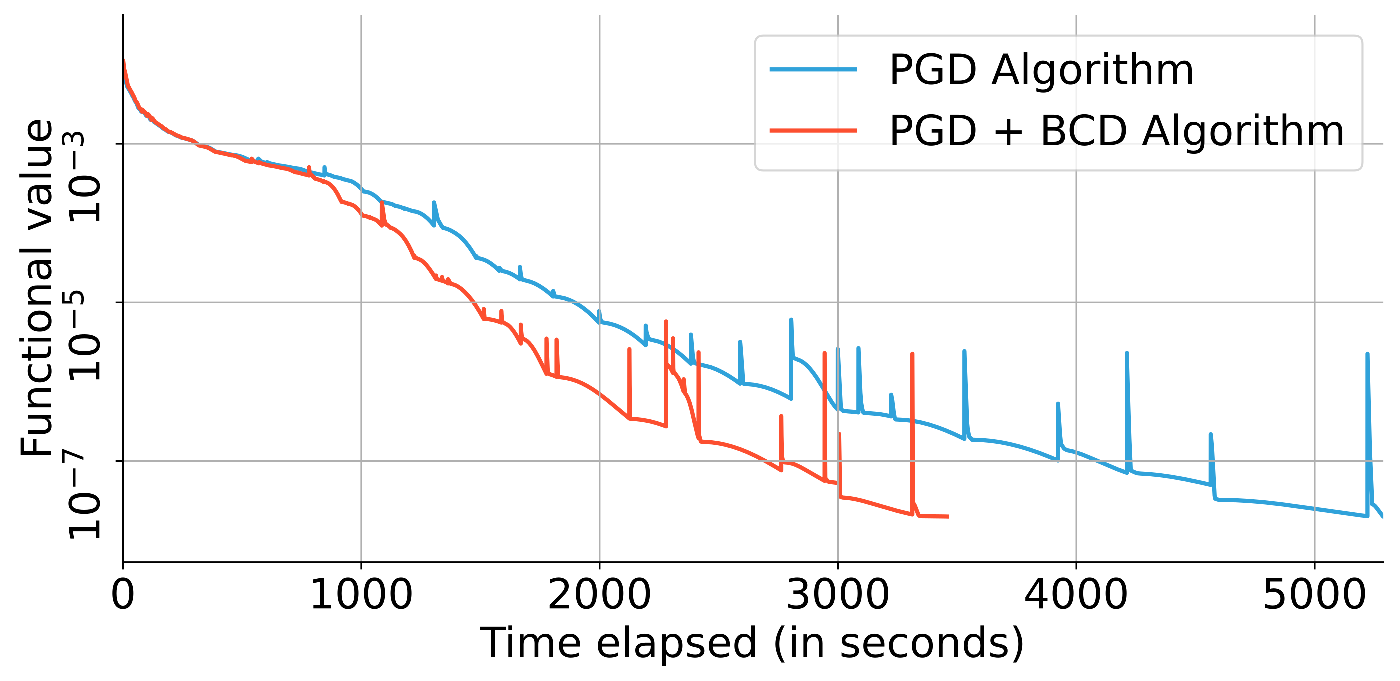}
		\label{fig:MATIRF_error_both_time}
		}
		\vspace{-.5\baselineskip}
		\caption{
			(a) Ground truth and the estimated signal by PGD with MA-TIRF,
			(b) Ground truth and the estimated signal by BCD with MA-TIRF,
			(c) Comparison norms of residues by PGD and BCD in function of time.}
		\label{fig:MATIRF_both}
	\end{figure}

	In Fig. \ref{fig:MATIRF_xesti_PGD}, we observe that the spikes are well recovered by PGD as expected.
	In Fig. \ref{fig:MATIRF_error_both_time} (in blue), we show the norm of the residue $\| y - Ax \|_{2}^{2}$ through the iterations for both PGD (in blue) and BCD (in red). Note that we stop the descent when the norm of the residue gets below $ 2 \times 10^{-8} $. For PGD, the norm of the residue decreases and reaches its minimum at $2 \times 10^{-8}$. The sudden jumps in the norm of the residue are due to the  projection used.

	In Fig. \ref{fig:MATIRF_xesti_BCD}, we obtain an estimated signal by BCD similar to the one obtained with PGD.
	In Fig. \ref{fig:MATIRF_error_both_time} (in red), as for PGD, we observe the norm of the residue through the iterations of BCD. We note that for the same norm of the residue, BCD converges much faster with an acceleration of almost $ 35\% $ in computation time.

	\begin{figure}[htbp]
	\centering
	\subfigure[]{
		\includegraphics[width=.45\columnwidth]{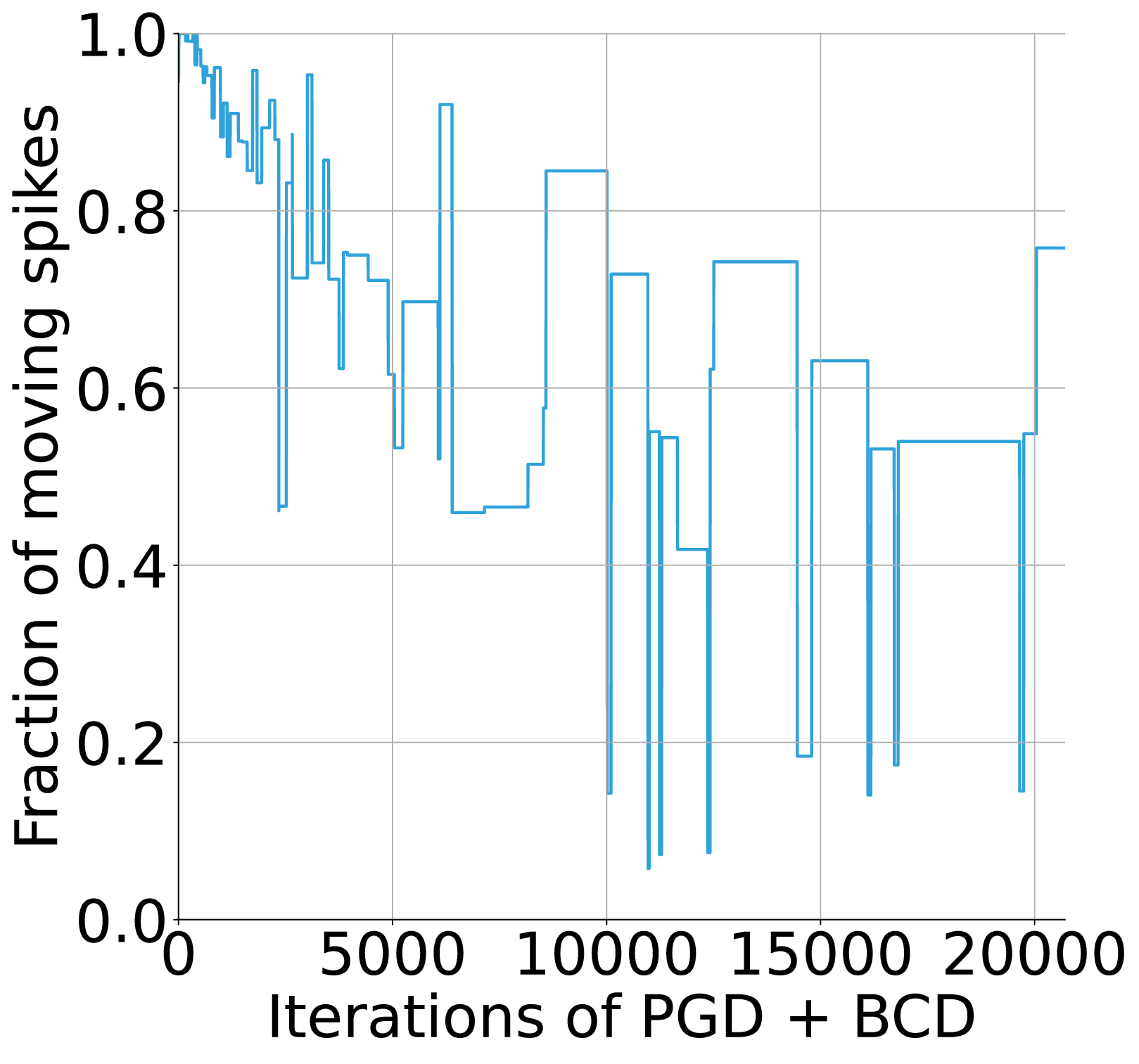}
		\label{fig:MATIRF_fraction_BCD}
	}
	\subfigure[]{
		\includegraphics[width=.48\columnwidth]{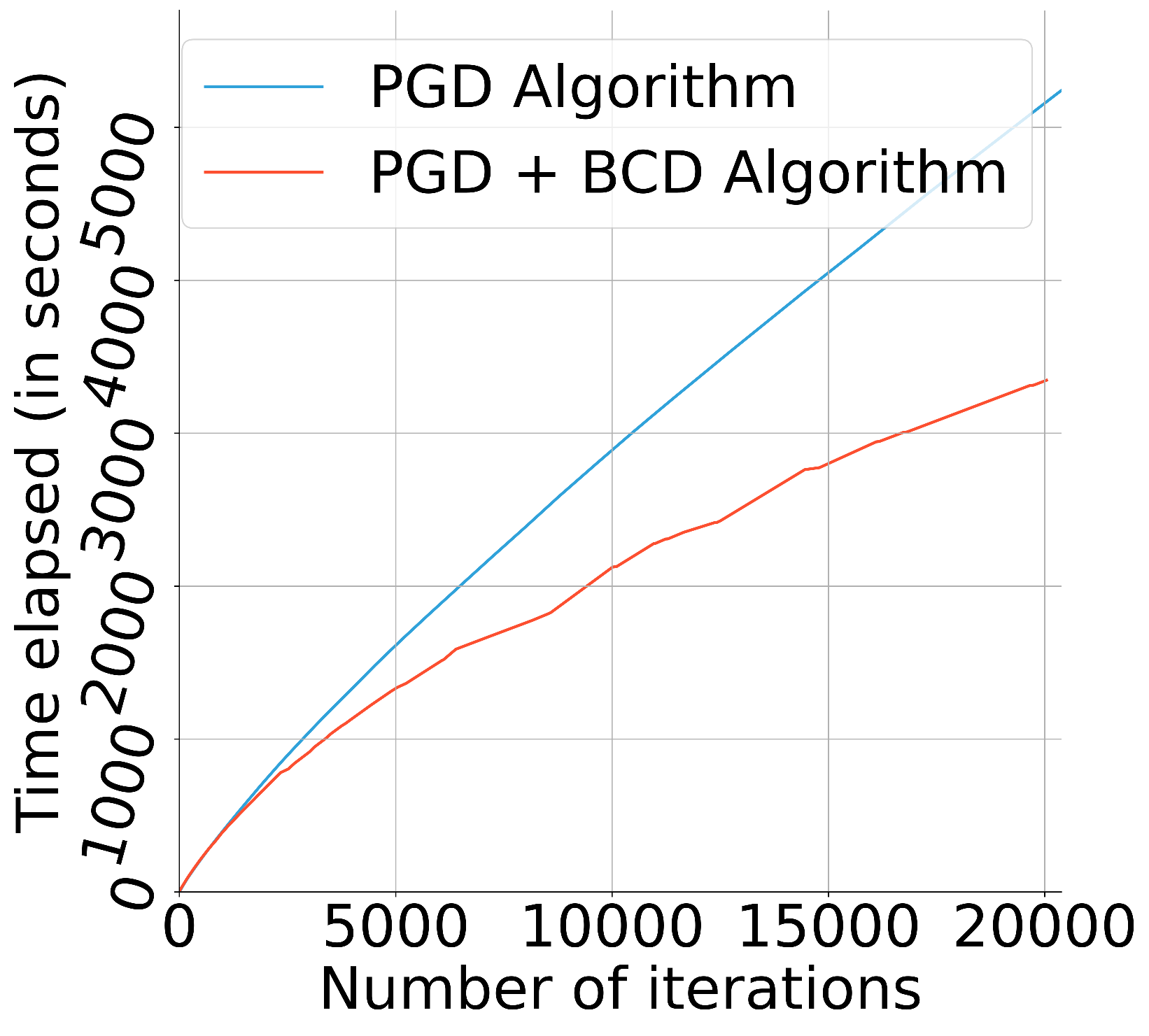}
		\label{fig:MATIRF_time_both}
	}
	\vspace{-.5\baselineskip}
	\caption{
		(a) Fraction of spikes whose gradients are computed during BCD,
		(b) Time elapsed for PGD and BCD in function of the number of iterations}
	\label{fig:MATIRF_misc}
	\end{figure}

	While the start of both methods takes the same time for about $ 2000 $ iterations, we observe a clear acceleration as the time passes. Indeed, as seen in Figure \ref{fig:MATIRF_fraction_BCD}, the fraction of gradients computed at each iteration of BCD decreases. A good empirical value for the choice of \texttt{threshold} in Algorithm \ref{alg:BCD} is $ 10^{-3} $. This value is used for all of our experiments. It provides a good trade-off in terms of minimization of $ g $
	(see Equation \eqref{eq_g})
	and computation time. Indeed, this threshold limits the number of blocks used for each \texttt{projected\_descent} by selecting only non-converged spikes. At the start of the descent, almost all gradients are still computed at each iteration, but as the iterations passes, the fraction of computed gradients decreases, hence the computation time gained over PGD. Both Figures \ref{fig:MATIRF_fraction_BCD} and \ref{fig:MATIRF_time_both} reveal that by computing fewer gradients, we can gain time on the total computation.

	After being tested on $10$ similar signals composed of $ 50 $ spikes with a high separation between spikes, the results stay the same and we have that in general, BCD is $35\%$ faster than PGD for MA-TIRF. Other metrics used to compare both estimations are the Jaccard, Precision and Recall. For these metrics, there is no particular difference between the two methods PGD and BCD.

	We note that for other kind of linear operators (such as the Fourier Transform), performances of the BCD can vary, but it always improves the computation speed.

	\section{Discussion/Conclusion}

	To conclude, we observe that the block structure for the sparse spike resolution permits to improve computation time of descent algorithms. Indeed, in a realistic microscopy setting, the time gained by BCD is up to $ 35\% $ when compared to previous methods. For future works, we intend to try other greedy methods for a better block selection. Moreover, one can parallelize the descent on the different blocks as it is a major benefit of Block Coordinate Descent methods.

	\section{Acknowledgments}

	This work was supported by the French National Research Agency (ANR) under reference ANR-20-CE40-0001 (EFFIREG project). Experiments presented in this paper were carried out using the PlaFRIM experimental testbed, supported by Inria, CNRS (LABRI and IMB), Université de Bordeaux, Bordeaux INP and Conseil Régional d’Aquitaine (see \url{https://www.plafrim.fr}).

	\bibliographystyle{abbrv}
	\bibliography{bibliography.bib}

\end{document}